\documentclass[12pt]{amsart}
\usepackage{fullpage}
\usepackage{amsfonts,amscd}
\usepackage{amssymb}
\usepackage{url}
\usepackage{graphicx}
\usepackage[english]{babel}
\theoremstyle{plain}
\newtheorem{theorem}                {Theorem}      [section]
\newtheorem{proposition}  [theorem]  {Proposition}

\newtheorem{lemma}        [theorem]  {Lemma}

\theoremstyle{definition}

\newtheorem{remark}       [theorem]  {Remark}

\numberwithin{equation}{section}

\def \R{{\mathbb R}}
\def \s{{\mathbb S}}

\usepackage{color}

\DeclareMathOperator{\trace}{trace}
\numberwithin{equation}{section}

\begin{document}

\title[]{New examples of $r$-harmonic immersions into the sphere}

\author{S.~Montaldo}
\address{Universit\`a degli Studi di Cagliari\\
Dipartimento di Matematica e Informatica\\
Via Ospedale 72\\
09124 Cagliari, Italia}
\email{montaldo@unica.it}
%
%\author{C.~Oniciuc}
%\address{Faculty of Mathematics\\ ``Al.I. Cuza'' University of Iasi\\
%Bd. Carol I no. 11 \\
%700506 Iasi, ROMANIA}
%\email{oniciucc@uaic.ro}

\author{A.~Ratto}
\address{Universit\`a degli Studi di Cagliari\\
Dipartimento di Matematica e Informatica\\
Viale Merello 93\\
09123 Cagliari, Italia}
\email{rattoa@unica.it}

\begin{abstract}
Polyharmonic, or $r$-harmonic, maps are a natural generalization of harmonic maps whose study was proposed by Eells-Lemaire in 1983. The main aim of this paper is to construct new examples of proper $r$-harmonic immersions into spheres. In particular, we shall prove that the canonical inclusion $ i\,: \, \s^{n-1}(R)\, \hookrightarrow \, \s^n$ is a proper $r$-harmonic submanifold of $\s^n$ if and only if the radius $R$ is equal to $1 \slash \sqrt{r}$. We shall  also prove the existence of proper $r$-harmonic generalized Clifford's tori into the sphere.
\end{abstract}



\thanks{Work supported by P.R.I.N. $2015$ -- Real and Complex Manifolds: Geometry, Topology and Harmonic Analysis -- Italy}

\maketitle

\section{Introduction}\label{intro}

{\it Harmonic maps} are the critical points of the {\em energy} functional
\begin{equation}\label{energia}
E(\varphi)=\frac{1}{2}\int_{M}\,|d\varphi|^2\,dv_M \,\, ,
\end{equation}
where $\varphi:M\to N$ is a smooth map between two Riemannian
manifolds $M$ and $N$. In analytical terms, the condition of harmonicity is equivalent to the fact that the map $\varphi$ is a solution of the Euler-Lagrange equation associated to the energy functional \eqref{energia}, i.e.
\begin{equation}\label{harmonicityequation}
  - d^* d \varphi =   {\trace} \, \nabla d \varphi =0 \,\, .
\end{equation}
The left member of \eqref{harmonicityequation} is a vector field along the map $\varphi$ or, equivalently, a section of the pull-back bundle $\varphi^{-1} \, (TN)$: it is called {\em tension field} and denoted $\tau (\varphi)$. Also, we recall the well-known fact that, if $\varphi$ is an \textit{isometric} immersion, then $\varphi$ is a harmonic map if and only $\varphi(M)$ is a minimal submanifold of $N$ (we refer to \cite{EL1, EL83} for background on harmonic maps).
A related topic of growing interest deals with the study of the so-called {\it polyharmonic maps}, or {\it $r$-harmonic maps}: these maps, which provide a natural generalisation of harmonic maps, are the critical points of the $r$-energy functional (as suggested in \cite{EL83}, \cite{ES})
\begin{equation}\label{bienergia}
    E_r(\varphi)=\frac{1}{2}\int_{M}\,|(d^*+d)^r (\varphi)|^2\,dv_M\,\, .
\end{equation}
In the case that $r=2$, the functional \eqref{bienergia} is called bienergy and its critical points are the so-called biharmonic maps. There have been extensive studies on biharmonic maps (see \cite{Jiang, SMCO} for an introduction to this topic and \cite{MOR1, Mont-Ratto2, Mont-Ratto3} for an approach which is related to this paper).
In 1989 Wang \cite{Wang} studied the first variational formula of the $r$-energy funtional \eqref{bienergia}, while the expression for its second variation was derived in \cite{Maeta3}, where it was shown that a biharmonic map is not always $r$-harmonic ($r \geq≥ 3$) and, more generally, that an $s$-harmonic map is not always $r$-harmonic ($s < r$). On the other hand, any harmonic map is trivially $r$-harmonic for all $r\geq 2$. Therefore, we say that an immersed submanifold into a Riemannian manifold $N$ is a {\it proper} $r$-harmonic submanifold if the immersion is an $r$-harmonic map which is \textit{not} harmonic (or, equivalently, \textit{not} minimal). 

As a general fact, when the ambient has nonpositive riemannian curvature tensor there are several results which assert that, under suitable conditions, an $r$-harmonic submanifold is minimal (see \cite{Chen} and \cite{Maeta4}, for instance), but the Chen conjecture that a biharmonic submanifold of $\R^n$ must be minimal is still open. On the other hand, when the target has positive curvature very little is known when $r \geq4$ and it is natural to look for new examples. In this order of ideas, the main aim of this paper is to produce new proper $r$-harmonic submanifolds of the Euclidean sphere. Our paper is organized as follows: in Section~\ref{section-Maeta} we recall some basic facts and fundamental formulas concerning $r$-harmonic maps. Next, in Section~\ref{proofs} we provide the proof of our main results which we now state (to fix notation, we shall denote by $\s^n(R)$ the Euclidean sphere of radius $R$ and write $\s^n$ for $\s^n(1)$).
\begin{theorem}\label{Corollary-parallel-spheres} Assume that $r,n \geq 2$. Then the canonical inclusion $ i\,: \, \s^{n-1}(R)\, \hookrightarrow \, \s^n$ is a proper $r$-harmonic submanifold of $\s^n$ if and only if the radius $R$ is equal to $1 \slash \sqrt{r}$.
\end{theorem}
\begin{theorem}\label{Corollary-parallel-spheres-Clifford} Let $r \geq 2,\,p,q \geq 1 $ and assume that the radii $R_1,R_2$ verify $R_1^2+R_2^2=1$. Then 
the generalized Clifford's torus inclusion $ i\,: \, \s^{p}(R_1) \times \s^{q}(R_2)\, \hookrightarrow \, \s^{p+q+1}$ is: 

(a) minimal if and only if

\begin{equation}\label{condizione-minimalita}
  R_1^2= \frac{p}{p+q} \,\,\quad {\rm and } \,\, \quad R_2^2 =\frac{q}{p+q}\,\, ; 
 \end{equation} 
 
 (b) a proper $r$-harmonic submanifold of $\s^{p+q+1}$ if and only if \eqref{condizione-minimalita} does not hold and either
\begin{equation}\label{caso:r=2}
 r=2  \,\, , \quad p \neq q \quad {\rm and } \,\, \quad R_1^2=R^2_2= \frac{1}{2}
 \end{equation} 
or $r \geq 3$ and $t=R_1^2$ is a root of the following polynomial: 
\begin{equation}\label{r-harmonicity-Clifford}
P(t)=r(p+q)\,t^3+[q-p-r(q+2p)]\,t^2+(2p+rp)\,t \,-\,p \,\,.
\end{equation}
\end{theorem}
\begin{remark}\label{remark-clifford2} Let $p = q$. Then $P(t)$ in \eqref{r-harmonicity-Clifford} takes the following form:
\begin{equation}\label{p=q}
P(t)=p \,(2t-1) \,(rt^2-rt+1) \,\, .
\end{equation}
Now, if $2 \leq r \leq 4$, then the only root is $t= (1\slash 2)$, which is associated to a minimal submanifold. But it is important to point out that, if $r \geq 5$, then $P(t)$ has two admissible solutions
$$
t= \frac{1}{2} \pm \frac{1}{2 } \,\sqrt {\frac{r-4}{r}}
$$
which give rise to proper $r$-harmonic generalized Clifford's tori in $\s^{2p+1}$.
\end{remark}
\begin{remark}\label{remark-clifford1} Let $p \neq q$. Then, for all $r \geq 3  , \,p,q \geq 1$, $P(t)$ has at least one admissible root $t^*$, i.e., such that $0<t^* <1$: this is a simple consequence of the fact that $P(0)=-p$ and $P(1)=q$. Moreover, it is easy to check that the associated submanifold cannot be minimal and so is a proper $r$-harmonic submanifold. By studying the sign of the discriminant of the third order polynomial $P(t)$, it can be shown that there are three distinct admissible solutions if 
\begin{equation}\label{condition-uniqueness}
(r^4-4r^3+24r^2-40 r-8)\,pq +4 (1-r)^3 \,(p^2+q^2) > 0 \,\,.
\end{equation}
In particular, if $p$ and $q$ are fixed, \eqref{condition-uniqueness} holds provided that $r$  is sufficiently large and so, in this case, we have three distinct proper $r$-harmonic generalized Clifford's tori in $\s^{p+q+1}$.
\end{remark}

\begin{remark}\label{remark-risultati-noti-paralleli} The results of Theorems~\ref{Corollary-parallel-spheres} and \ref{Corollary-parallel-spheres-Clifford} were known when $r=2$ (see \cite{Jiang} and \cite{CMO02}) and $r=3$ (see \cite{Maeta2})(to recover Maeta's result, set $\lambda^2= \cot^2 (\alpha^*)$ in Theorem $4.6$ of \cite{Maeta2}).
\end{remark}
\section{Generalities on $r$-harmonic maps}\label{section-Maeta}

Here we recall some basic facts whose proofs can be found in \cite{EL83} and \cite{Maeta1}. Let $\varphi:(M,g_M)\to(N,g_N)$ be a smooth map between two Riemannian manifolds $M$ and $N$ of dimension $m$ and $n$ respectively. Then $d \varphi$ is a $1$-form with values in the vector bundle $\varphi ^{-1}TN$ or, equivalently, a section of $T^*M\otimes \varphi ^{-1}TN$. In local charts $d \varphi$ is described by:
\begin{equation}\label{d-phi-inlocalcharts}
d \varphi= \frac{\partial \varphi^{\gamma}}{\partial x_i} \, dx^i \otimes \frac{\partial }{\partial y_\gamma}= \varphi^{\gamma}_i \,\, dx^i \otimes \frac{\partial }{\partial y_\gamma}  \,\, ,
\end{equation}
where, here and below, the Einstein's sum convention over repeated indices is adopted.
%
%Also, note that latin letters are used for indices related to the domain, while greek letters are associated to the codomain.
%
The second fundamental form $\nabla d \varphi$ is a covariant differentiation of the $1$-form $d \varphi$, i.e., a section of $\odot^2 T^*M\otimes \varphi ^{-1}TN$. The local coordinates expression for the second fundamental form is
\begin{equation}\label{seconda-forma-coordinate-locali}
\nabla d \varphi = \left ( \nabla d \varphi \right )^\gamma_{ij} \, dx^i dx^j \otimes\frac{\partial}{\partial y_\gamma} \,\, ,
\end{equation}
where
\begin{eqnarray}\label{seconda-forma-coordinate-locali-bis}
\left ( \nabla d \varphi \right )^\gamma_{ij} &=& \left [ \nabla_{\partial \slash \partial x_i} \left ( \varphi^{\beta}_\ell \,\, dx^\ell \otimes \frac{\partial }{\partial y_\beta} \right ) \frac{\partial}{\partial x_j} \right ]^\gamma\\ \nonumber
&=& \varphi^\gamma _{ij} - {}^M \Gamma^k_{ij} \varphi^\gamma _{k} + {}^N \Gamma^\gamma_{\beta \delta} \varphi^\beta _{i} \varphi^\delta _{j} \,\, . 
\end{eqnarray}
Now, since $\tau (\varphi)= - d^* d \varphi$ is the trace of the second fundamental form, its    description in local coordinates is
\begin{equation}\label{tau-coordinate-locali}
( \tau (\varphi) )^\gamma  = (- d^* d \varphi)^\gamma=g_M^{ij} \, \left ( \nabla d \varphi \right )^\gamma_{ij}\,\, .
\end{equation}

For our purposes, it will also be important to recall the definitions of the rough laplacian and of the sectional curvature operator. We shall denote by $\nabla^M, \nabla^N$ and $\nabla^{\varphi}$ the induced connections on the bundles $TM, TN$ and $\varphi ^{-1}TN$ respectively. The \textit{rough Laplacian} on sections of $\varphi^{-1} \, (TN)$, denoted by $\overline{\Delta}$, is defined by
\begin{equation}\label{roughlaplacian}
    \overline{\Delta}=d^* d =-\sum_{i=1}^m\{\nabla^{\varphi}_{e_i}
    \nabla^{\varphi}_{e_i}-\nabla^{\varphi}_
    {\nabla^M_{e_i}e_i}\}\,\,,
\end{equation}
where $\{e_i\}_{i=1}^m$ is a local orthonormal frame on $M$. The sectional curvature operator on $N$ is the $(1,3)$-tensor defined by:
\begin{equation}\label{curvatura}
    R^N (X,Y)W= \nabla^N_X \nabla^N_Y \, W- \nabla^N_Y \nabla^N_X \, W-\nabla^N_{[X,Y]}\,W \,\, .
\end{equation}

We now proceed to a general description of the $r$-energy \eqref{bienergia} when $r \geq 2$. If $r=2s$, $s \geq 1$:
\begin{eqnarray}\label{2s-energia}
E_{2s}(\varphi)&=& \frac{1}{2} \int_M \, \langle \, \underbrace{(d^* d) \ldots (d^* d)}\varphi, \,\underbrace{(d^* d) \ldots (d^* d)}\varphi \, \rangle_{_N}\, \,dv_M \\ \nonumber
&& \quad \qquad \qquad s\,\,{\rm \,times} \quad \qquad \,\,\quad s\,\,\,{\rm times}\\ \nonumber
&& \\ \nonumber
&=& \frac{1}{2} \int_M \, \langle \,\overline{\Delta}^{s-1}\tau(\varphi), \,\overline{\Delta}^{s-1}\tau(\varphi)\,\rangle_{_N} \, \,dv_M 
\end{eqnarray}
Now, the map $\varphi$ is $2s$-harmonic if, for all variations $\varphi_t$,
$$
\left .\frac{d}{dt} \, E_{2s}(\varphi_t) \, \right |_{t=0}\,=\,0 \,\,.
$$
Setting
$$
V= \left .\frac{\partial  \varphi_t}{ \partial t}  \,\right |_{t=0} \quad \in \,\, \Gamma (\varphi^{-1}TN) \,\,,
$$
\begin{equation}\label{r-energia-tension}
\left .\frac{d}{dt} \, E_{2s}(\varphi _t) \,\right |_{t=0}\,= \,-\,\int_M \, \langle\,\tau_{2s}(\varphi),V\,\rangle_{_N}\, \,dv_M \,\,,
\end{equation}
where the explicit formula for the $2s$-tension field $\tau_{2s}(\varphi)$ is:
\begin{eqnarray}\label{2s-tension}
\tau_{2s}(\varphi)&=&\overline{\Delta}^{2s-1}\tau(\varphi)-R^N \left(\overline{\Delta}^{2s-2} \tau(\varphi), d \varphi (e_j)\right ) d \varphi (e_j) \nonumber\\ 
&&  -\, \sum_{\ell=1}^{s-1}\, \left \{R^N \left( \nabla^\varphi_{e_j}\,\overline{\Delta}^{s+\ell-2} \tau(\varphi), \overline{\Delta}^{s-\ell-1} \tau(\varphi)\right ) d \varphi (e_j)  \right .\\ \nonumber
&& \qquad \qquad  -\, \left . R^N \left( \overline{\Delta}^{s+\ell-2} \tau(\varphi),\nabla^\varphi_{e_j}\, \overline{\Delta}^{s-\ell-1} \tau(\varphi)\right ) d \varphi (e_j)  \right \} \,\, ,
\end{eqnarray}
where $\overline{\Delta}^{-1}=0$ and $\{e_j\}_{j=1}^m$ is a local orthonormal frame on $M$ (here and below, the sum over $j$ is not written but understood). Of course, $\varphi$ is $2s$-harmonic if $\tau_{2s}(\varphi)$ vanishes identically. In the case that $r=2s+1$, the relevant modifications are:
\begin{eqnarray}\label{2s+1-energia}
E_{2s+1}(\varphi)&=& \frac{1}{2} \int_M \, \langle\,d\underbrace{(d^* d) \ldots (d^* d)}\varphi, \,d\underbrace{(d^* d) \ldots (d^* d)}\varphi\,\rangle_{_N}\, \,dv_M \nonumber\\ 
&& \,\quad \qquad \qquad s\,\,{\rm \,times} \quad \qquad \,\,\quad \, s\,\,\,{\rm times}\\ \nonumber
&& \\ \nonumber
&=& \frac{1}{2} \int_M \, \langle\,\nabla^\varphi_{e_j}\, \overline{\Delta}^{s-1}\tau(\varphi), \,\nabla^\varphi_{e_j}\,\overline{\Delta}^{s-1}\tau(\varphi)\, \rangle_{_N} \, \,dv_M \,\,\, ;
\end{eqnarray}
\begin{eqnarray}\label{2s+1-tension}
\tau_{2s+1}(\varphi)&=&\overline{\Delta}^{2s}\tau(\varphi)-R^N \left(\overline{\Delta}^{2s-1} \tau(\varphi), d \varphi (e_j)\right ) d \varphi (e_j)\nonumber \\ 
&&  -\, \sum_{\ell=1}^{s-1}\, \left \{R^N \left( \nabla^\varphi_{e_j}\,\overline{\Delta}^{s+\ell-1} \tau(\varphi), \overline{\Delta}^{s-\ell-1} \tau(\varphi)\right ) d \varphi (e_j)  \right .\\ \nonumber
&& \qquad \qquad  -\, \left . R^N \left( \overline{\Delta}^{s+\ell-1} \tau(\varphi),\nabla^\varphi_{e_j}\, \overline{\Delta}^{s-\ell-1} \tau(\varphi)\right ) d \varphi (e_j)  \right \} \\ \nonumber
&& \,-\,R^N \left( \nabla^\varphi_{e_j}\,\overline{\Delta}^{s-1} \tau(\varphi), \overline{\Delta}^{s-1} \tau(\varphi)\right ) d \varphi (e_j)\,\,. 
\end{eqnarray}
We point out that $\tau_r(\varphi)=0$ is a semi-linear, elliptic PDE's system of order $2r$.
%%%%%%%%%%%%%%
%%%%%%%%%%%%%%%

\section{Proofs of Theorems~\ref{Corollary-parallel-spheres} and \ref{Corollary-parallel-spheres-Clifford}}\label{proofs}
In order to prove Theorem~\ref{Corollary-parallel-spheres} it is sufficient to determine the condition of $r$-harmonicity for a map defined as follows:
\begin{equation}\label{rotationallysymmetricmaps}\begin{array}{lcll}
                                           \varphi_{\alpha^*} \,:\, &\s^{n-1} &\to &  \s^n \subset \R^n \times \R \\
                                           &&& \\
                                           &w &\mapsto & ( \sin \alpha^*\,w, \,\cos \alpha^*) \,\, ,
                                         \end{array}
\end{equation}
where $\alpha^*$ is a fixed constant value in the interval $(0,\pi)$. Indeed, if $\varphi_{\alpha^*}$ is a map as in \eqref{rotationallysymmetricmaps}, then the induced metric on $\s^{n-1}$ is given by $[\varphi_{\alpha^*}]^* (g_{\s^n})= (\sin^2 \alpha^*)\,g_{\s^{n-1}} $. Therefore, since $r$-harmonicity is preserved by multiplication of the riemannian metric of the domain manifold by a positive constant, we conclude that if $\varphi_{\alpha^*}$ is a proper $r$-harmonic map, then its image $\varphi_{\alpha^*}(\s^{n-1})= \s^{n-1}(\sin \alpha^*)$ is a proper $r$-harmonic hypersphere of radius $R=\sin \alpha^*$. By way of summary, we only have to prove that the map $\varphi_{\alpha^*}$ in \eqref{rotationallysymmetricmaps} is a proper $r$-harmonic map if and only if $\sin \alpha^* = 1 \slash \sqrt r$. This is an immediate consequence of the following:

\begin{proposition}\label{main-theor-parallels} Let $\varphi_{\alpha^*}\,: \,\,\s^{n-1} \to \,\s^n$ be a map of the type \eqref{rotationallysymmetricmaps}. Then $\varphi_{\alpha^*}$ is a proper $r$-harmonic map ($r \geq 2$) if and only if $\cos \alpha^* \neq 0$ and $\alpha^*$ is a critical point of the function $\varepsilon_r \, : (0,\pi) \to \R $ defined by
\begin{equation}\label{condizione-r-harmonicity}
 \varepsilon_r (\alpha)= \sin^2\alpha \, \cos^{2(r-1)}\alpha \,\, .
\end{equation}
\end{proposition}
\begin{remark}\label{remark-dopo-main-theorem} We shall see that, up to the constant 
$$c=\frac{1}{2}\,(n-1)^r \, {\rm Vol}(\s^{n-1}) \,\, , 
$$
$\varepsilon_r (\alpha^*)$ coincides with the $r$-energy of the map $\varphi_{\alpha^*}$ in \eqref{rotationallysymmetricmaps} ($r \geq 1$). In particular, if $\cos \alpha^* = 0$, then $\varphi_{\alpha^*}$ is harmonic.
\end{remark}
\begin{remark}\label{remark-stability} If $\alpha^*$ is a critical point of the function $\varepsilon_r (\alpha)$ defined in \eqref{condizione-r-harmonicity}, then it is easy to check that $\varepsilon_r''\, (\alpha^*)<0$: as a simple consequence of this fact, we deduce that all the $r$-harmonic hyperspheres obtained in Theorem \ref{Corollary-parallel-spheres} are \textit{unstable} critical points.
\end{remark}
The proof of Proposition~\ref{main-theor-parallels} is based on a series of lemmata in which we compute the relevant covariant derivatives. To this purpose, it is convenient to carry out a specific preliminary work. Let $w_1,\, \ldots ,\,w_{n-1}$ be a set of local coordinates  on the domain $\s^{n-1}$. On $\s^n$ we choose coordinates $w_j, \alpha$ so that the metric tensor is described by
\begin{equation}\label{metric-s^n}
g_{\s^n}=\sin^2 \alpha \,\, g_{\s^{n-1}} + d \alpha^2\,\,.
\end{equation}
We observe that $d \varphi_{\alpha^*}\left ( \partial \slash \partial w_j \right )=\partial \slash \partial w_j$ for all $1 \leq j \leq n-1$ and so, if it is clear from the context, we will not state explicitly whether a vector field $\partial \slash \partial w_i$ is to be considered on the domain or on the codomain. The computation of several covariant derivatives below is based on
\begin{lemma}\label{lemma-christoffels} Let $w_1, \ldots,w_{n-1},\alpha$ be local coordinates on $\s^n$ as above. Then their associated Christoffel's symbols $\Gamma^k_{ij}$ are described by the following table:
\begin{equation}\label{simb-christ}
\begin{array}{lll}
{\rm (i)}&{\rm If}\, 1 \leq i,j,k \leq n-1: & \Gamma^k_{ij}={}^{\s}\Gamma^k_{ij} \\
{\rm (ii)}&{\rm If}\, 1 \leq i, j \leq n-1: & \Gamma^n_{ij}=\, -\,\sin \alpha \, \cos \alpha \,\,\,\,(g_{\s})_{ij} \\
{\rm (iii)}&{\rm If}\, 1 \leq i,j \leq n-1: & \Gamma^j_{i n}=\frac{\cos \alpha}{\sin \alpha} \,\,\delta_i^j \\
{\rm (iv)}&{\rm If}\, 1 \leq j \leq n: & \Gamma^j_{nn}=0=\Gamma^n_{jn} \,\,,\\
\end{array}
 \end{equation}
where ${}^{\s}\Gamma^k_{ij}$ and $g_{\s}$ denote the Christoffel symbols and the metric tensor of $\s^{n-1}$ respectively.
\end{lemma}
\begin{proof}The proof is a straightforward computation based on the well-known formula
\begin{equation}\label{formula-simboli}
\Gamma^k_{ij}= \frac{1}{2}\, g^{k \ell} \left(\frac{\partial g_{j \ell}}{\partial y_i}+ \frac{\partial g_{ \ell i}}{\partial y_j}-\frac{\partial g_{ij}}{\partial y_\ell} \right)\,\, .
\end{equation}
\end{proof}
The next lemma is elementary:
\begin{lemma}\label{lemma-tau} Let $\varphi_{\alpha^*}$ be a map as in \eqref{rotationallysymmetricmaps}. Then
\begin{equation}\label{tauvarphi-alpha*}
\tau(\varphi_{\alpha^*}) = \,F(\alpha^*) \frac{\partial}{\partial \alpha}
\end{equation}
where we have set
\begin{equation}\label{definizionediF}
F(\alpha^*) = - \,(n-1) \, \sin \alpha^*\, \cos \alpha^* \,\, .
\end{equation}
\end{lemma}
Now we prove three lemmata which will play a key role:
\begin{lemma}\label{lemma-dtau}Let $\varphi_{\alpha^*}$ be a map as in \eqref{rotationallysymmetricmaps}. Then
\begin{equation}\label{dtauvarphi-alpha*}
d\tau(\varphi_{\alpha^*}) = G(\alpha^*) \,\, \sum_{i=1}^{n-1}\,\,dw^i \otimes \frac{\partial}{\partial w_i} \,\, ,
\end{equation}
where we have set
\begin{equation}\label{definizionediG}
G(\alpha^*) = F(\alpha^*)\,\, \frac{\cos \alpha^*}{\sin \alpha^*}=\,- \,(n-1)\,\cos^2 \alpha^* \,\, .
\end{equation}
\end{lemma}
\begin{proof} $d\tau(\varphi_{\alpha^*})$ is a section of $T^*M\otimes \varphi ^{-1}TN$.  Equation \eqref{dtauvarphi-alpha*} tells us that we just have to verify that
\begin{equation}\label{dtauvarphi-alpha*-bis}
\left (d\tau(\varphi_{\alpha^*})\right )^i_i =\,F(\alpha^*)\,\, \frac{\cos \alpha^*}{\sin \alpha^*} \qquad \quad {\rm if } \,\, 1 \leq i \leq n-1
\end{equation}
and
\begin{equation}\label{dtauvarphi-alpha*-tris}
\left (d\tau(\varphi_{\alpha^*})\right )^i_j = 0\qquad \quad {\rm whenever } \,\, i \neq j \,\,.
\end{equation}
Now, \eqref{dtauvarphi-alpha*-bis} and \eqref{dtauvarphi-alpha*-tris} are a simple consequence of the following:
\begin{eqnarray}\label{derivatadidtau}
d\tau(\varphi_{\alpha^*})\left( \frac{\partial}{\partial w_i} \right)&=& \nabla^\varphi_{\partial \slash \partial w_i}\,\tau(\varphi_{\alpha^*})\nonumber \\ 
&=&\nabla^\varphi_{\partial \slash \partial w_i}\,F(\alpha^*) \frac{\partial}{\partial \alpha}= F(\alpha^*)\,\nabla^\varphi_{\partial \slash \partial w_i}\, \frac{\partial}{\partial \alpha} \\ \nonumber
&=&F(\alpha^*)\,\left[ \Gamma^k_{in}\,\,\frac{\partial}{\partial w_k}+\, \Gamma^n_{in}\,\,\frac{\partial}{\partial \alpha}\right ]\\ \nonumber
&=&F(\alpha^*)\,\,\frac{\cos \alpha^*}{\sin \alpha^*}\,\,\delta^k_i  \,\frac{\partial}{\partial w_k}=F(\alpha^*)\,\frac{\cos \alpha^*}{\sin \alpha^*} \,\frac{\partial}{\partial w_i} \,\, ,
\end{eqnarray}
where, in \eqref{derivatadidtau}, we have used the explicit expression of the Christoffel symbols given in Lemma \ref{simb-christ}.
\end{proof}

\begin{lemma}\label{lemma-d*dtau}Let $\varphi_{\alpha^*}$ be a map as in \eqref{rotationallysymmetricmaps}. Then
\begin{equation}\label{d*dtauvarphi-alpha*}
d^*d\left(\tau(\varphi_{\alpha^*})\right ) = H(\alpha^*) \,\,  \frac{\partial}{\partial \alpha} \,\, ,
\end{equation}
where
\begin{equation}\label{definizionediH}
H(\alpha^*) = (n-1)\, G(\alpha^*)\,\, \sin \alpha^* \, \cos \alpha^* \,\, .
\end{equation}
\end{lemma}
\begin{proof}By using Lemma \ref{simb-christ} and its notation, we compute (see \cite{EL83}):
\begin{eqnarray}\label{formula-lemma}
\sum_{\ell=1}^{n-1}\left (\nabla_{\partial \slash \partial w_i} \left( dw^\ell \otimes \frac{\partial }{\partial w_\ell}\right )\right)&=&-\,{}^{\s}\Gamma^{\gamma}_{ik}\, dw^k \otimes \frac{\partial }{\partial w_\gamma} \nonumber \\ 
&& +\,\Gamma^{\beta}_{i\gamma}\,dw^\gamma \otimes \frac{\partial }{\partial w_\beta}
+\,\Gamma^{n}_{i\gamma}\,dw^\gamma \otimes \frac{\partial }{\partial \alpha} \\ \nonumber
&=&\Gamma^{n}_{i\gamma}\,dw^\gamma \otimes \frac{\partial }{\partial \alpha}= - \sin \alpha^* \cos \alpha^* \,(g_{\s})_{i \gamma}\,dw^\gamma \otimes \frac{\partial }{\partial \alpha} \,\,.
\end{eqnarray}
Next, by using Lemma \ref{lemma-dtau} and \eqref{formula-lemma}, we obtain
\begin{eqnarray}
d^*d\left(\tau(\varphi_{\alpha^*})\right ) &=&-\, (g_{\s}^{-1})^{ij}\,\sum_{i,\ell=1}^{n-1}\, \left ( G(\alpha^*)\,
\nabla_{\partial \slash \partial w_i} \left( dw^\ell \otimes \frac{\partial }{\partial w_\ell}\right ) \,\frac{\partial }{\partial w_j}\right ) \nonumber \\
&=& (n-1)\,G(\alpha^*)\,\, \sin \alpha^* \, \cos \alpha^* \,\frac{\partial }{\partial \alpha}\,\,,  \\ \nonumber
\end{eqnarray}
so ending the proof of this lemma.
\end{proof}
For future use, we also note that \eqref{definizionediH} can be rewritten as:
$$
H(\alpha^*) = (n-1)\, G(\alpha^*)\sin \alpha^* \, \cos \alpha^* =- \,(n-1)^2\,F(\alpha^*) \,\cos^2 \alpha^* \,\, ,
$$
where $F(\alpha^*)$ is as in \eqref{definizionediF}.
\begin{lemma}\label{lemma-r-energy}Let $r \geq2$ and $\varphi_{\alpha^*}$ be a map as in \eqref{rotationallysymmetricmaps}. Then its $r$-energy is given by:
\begin{equation}\label{d*dtauvarphi-alpha*-energy}
E_r(\varphi_{\alpha^*})= \, \frac{1}{2}\,{\rm Vol}\,(\s^{n-1})\,(n-1)^r\,\varepsilon_r (\alpha^*) \,\, ,
\end{equation}
where $\varepsilon_r (\alpha)$ is the function defined in \eqref{condizione-r-harmonicity}.
\end{lemma}
\begin{proof} The proof reduces to an iteration of the calculations which we have performed in Lemmata \ref{lemma-tau}, \ref{lemma-dtau}, \ref{lemma-d*dtau}. We have:
\begin{eqnarray}\label{E2-lemma}
E_2(\varphi_{\alpha^*})&=&\frac{1}{2}\,\int_{\s^{n-1}} \, |\tau(\varphi_{\alpha^*})|^2\, dv_{\s^{n-1}}= \frac{1}{2}\, {\rm Vol}\,(\s^{n-1})\,\, F^2(\alpha^*)\nonumber\\ 
&=&\frac{1}{2}\, {\rm Vol}\,(\s^{n-1})\,\,(n-1)^2 \,\, \varepsilon_2 (\alpha^*) \,\, .
\end{eqnarray}
\begin{eqnarray}\label{E3-lemma}
E_3(\varphi_{\alpha^*})&=& \frac{1}{2}\,\int_{\s^{n-1}} \, |d\tau(\varphi_{\alpha^*})|^2\, dv_{\s^{n-1}}\nonumber\\
&=& \frac{1}{2}\, {\rm Vol}\,(\s^{n-1})\,\,(n-1) \,(\sin^2 \alpha^*)\,G^2(\alpha^*) \nonumber \\ 
&=&\frac{1}{2}\, {\rm Vol}\,(\s^{n-1})\,\,(n-1)\,(\cos^2 \alpha^*)\,F^2(\alpha^*) \\ 
&=&\frac{1}{2}\, {\rm Vol}\,(\s^{n-1})\,\,(n-1)^3 \,\, \varepsilon_3 (\alpha^*) \,\, .\nonumber
\end{eqnarray}
%&=&\frac{1}{2}\,\int_{\s^{n-1}} \, G^2(\alpha^*)\, (g^{-1}_{\s^{n-1}})^{ij}\,\delta_i^\beta \,\delta_j^\gamma \,(g_{\s^{n}})_{\beta \gamma}\,\, dv_{\s^{n-1}}\nonumber\\ 
\begin{eqnarray}\label{E4-lemma}
E_4(\varphi_{\alpha^*})&=& \, \frac{1}{2}\,\int_{\s^{n-1}} \, |d^*d \left (\tau(\varphi_{\alpha^*})\right )|^2\, dv_{\s^{n-1}}= \frac{1}{2}\, {\rm Vol}\,(\s^{n-1})\,\, H^2(\alpha^*)\nonumber\\ 
&=&\frac{1}{2}\, {\rm Vol}\,(\s^{n-1})\,\,(n-1)^2\, ( \sin^2 \alpha^*) \, ( \cos^2 \alpha^*)\,
G^2(\alpha^*) \\ \nonumber
&=&\frac{1}{2}\, {\rm Vol}\,(\s^{n-1})\,\,(n-1)^2\, (\cos^4 \alpha^*)\, F^2(\alpha^*) \\ \nonumber
&=&\frac{1}{2}\, {\rm Vol}\,(\s^{n-1})\,\,(n-1)^4 \,\, \varepsilon_4 (\alpha^*) \,\, .
\end{eqnarray}
Now the iterative procedure can be made explicit and we recognize the pattern
\begin{equation}\label{pattern-iteration}
E_{r+1}(\varphi_{\alpha^*})=(n-1)\, \cos^2 \alpha^* \,\,E_{r}(\varphi_{\alpha^*})
\end{equation}
from which the conclusion follows by induction.
\end{proof}
\begin{remark} The conclusion of Lemma \ref{lemma-r-energy} is also true for $r=1$, but this is not of interest for our purposes.
\end{remark}
Now we can end the proof of Proposition \ref{main-theor-parallels}:
\begin{proof} According to Lemma \ref{lemma-r-energy}, $\varepsilon_r({\alpha^*})$ is, up to a constant, the $r$-energy of $\varphi_{\alpha^*}$. Also, we observe that $G={\rm SO}(n)$ acts naturally by isometries on both the domain and the codomain of $\varphi_{\alpha^*}$, and $\varphi_{\alpha^*}$ is $G$-equivariant. Therefore, if $\alpha^*$ is a critical point of $\varepsilon_r({\alpha})$, then $\varphi_{\alpha^*}$ is a critical point of the $r$-energy functional with respect to equivariant variations. But the $r$-energy functional is invariant by isometries and therefore the principle of symmetric criticality of Palais (\cite{Palais}) can be applied to conclude that actually $\varphi_{\alpha^*}$ is a critical point of the $r$-energy functional with respect to all variations. However, in order to avoid the use of Palais's general principle, we provide here an alternative, more explicit way to end the proof. Namely, we verify directly that \begin{equation}\label{tau-ben-messo}
    \tau_r(\varphi_{\alpha^*}) =\overline{T_r}(\alpha^*) \,\, \frac{\partial}{\partial \alpha} \quad {\rm for \,\, all }\quad r \geq2
\end{equation}
for some functions $\overline{T_r}(\alpha)$, a fact which is clearly sufficient to end the proof. To this purpose, we first observe that, computing as in Lemma \ref{lemma-d*dtau}, we have
\begin{equation}\label{Delta-iterato-ok}
\overline{\Delta}^s \, \tau(\varphi_{\alpha^*}) = \overline{H}_s(\alpha^*)\, \frac{\partial}{\partial \alpha}
\end{equation}
for some $\overline{H}_s(\alpha^*)$ whose explicit expressions play no role in the sequel. Moreover, let $\left \{ e_j \right \}_{j=1}^{n-1}$ be a local orthonormal frame on $\s^{n-1}$ (we have $d\varphi_{\alpha^*}(e_j)=e_j$ but note that, when $e_j$ is considered as a tangent vector to $\s^n$, $\langle e_j,e_j \rangle= \sin^2 \alpha  \,\, {\rm and} \,\,  \langle e_j,(\partial \slash \partial \alpha)\rangle=0 
$). By writing $e_j$ as a linear combination of the coordinate frame fields $\partial \slash \partial w_i$'s and using Lemma \ref{simb-christ} it is easy to check that
\begin{equation}\label{nabla-Delta-iterato-ok}
\nabla^\varphi_{e_j}\, \left (\overline{\Delta}^s \, \tau(\varphi_{\alpha^*}) \right ) = \overline{G}_s(\alpha^*) \,\, e_j
\end{equation}
for some $\overline{G}_s(\alpha^*)$ whose explicit expressions play no role in the sequel. Let us denote by $R(X,Y)Z$ the Riemann sectional curvature tensor of the codomain $\s^n$. Inserting the information \eqref{Delta-iterato-ok}, \eqref{nabla-Delta-iterato-ok} into the explicit formulas for $\tau_r$ (see \eqref{2s-tension} and \eqref{2s+1-tension}) we conclude that we only need to check that both
\begin{equation}\label{richieste-su-R}
{\rm (i)}\quad R \left(\frac{\partial}{\partial \alpha} ,e_j \right)\,e_j \qquad {\rm (ii)}\quad R \left(e_j,\frac{\partial}{\partial \alpha}   \right)\,e_j
\end{equation}
have the form $c(\alpha^*) \,\partial \slash \partial \alpha$, where $c(\alpha^*)$ is a real number which depends only on $\alpha^*$ (because of the symmetries of $R$, it actually suffices to verify this claim in the case \eqref{richieste-su-R}(i)). To this purpose, by using the standard expression for the curvature operator on the sphere we obtain:
$$
R \left(\frac{\partial}{\partial \alpha} ,e_j  \right)\,e_j = \langle e_j,e_j \rangle \,\frac{\partial}{\partial \alpha} -\langle e_j,\frac{\partial}{\partial \alpha}  \rangle \,e_j= (\sin \alpha^*)^2\, \, \frac{\partial}{\partial \alpha}  \,\, ,
$$
so ending the proof.
\end{proof}
\begin{remark}\label{remark-tau-esplicito} It is actually possible to have the explicit expression of the $r$-tension field \eqref{tau-ben-messo} by means of the following argument. Let us consider the following variation:
\begin{equation}\label{variation-good}
    \varphi_{\alpha^*,t}= (w,\alpha^*+t) \,\,.
\end{equation}
Clearly,
\begin{equation}\label{variation-good-bis}
  \left . \frac{\partial \varphi_{\alpha^*,t}}{\partial t}\right |_{t=0}= \frac{\partial}{\partial \alpha}\,\,.
\end{equation}
Then, because of \eqref{r-energia-tension} (which also holds when $r$ is odd), \eqref{tau-ben-messo} yields:
\begin{eqnarray}\label{variation-good-tris}
\left .\frac{d}{dt}\, E_{r}(\varphi_{\alpha^*,t}) \right |_{t=0}\,&=& \,-\,\int_{\s^{n-1}} \, \langle \,\tau_{r}(\varphi_{\alpha^*}), \frac{\partial}{\partial \alpha}\, \rangle_{_N}\, \,dv_{\s^{n-1}}  \\ \nonumber
&=& -\,{\rm Vol}(\s^{n-1}) \, \overline{T_r}(\alpha^*) \,\, . 
\end{eqnarray}
On the other hand, direct calculation by using \eqref{d*dtauvarphi-alpha*-energy} yields:
\begin{eqnarray}\label{variation-good-tris+1}
\left .\frac{d}{dt} \, E_{r}(\varphi_{\alpha^*,t}) \,\right |_{t=0}&=& \,\frac{1}{2}\,\int_{\s^{n-1}} \,(n-1)^r \, \left .\frac{d\,\varepsilon_r(\alpha^*+t)}{dt}\right |_{t=0}\,dv_{\s^{n-1}} \nonumber \\ 
&=& \,\frac{1}{2}\,{\rm Vol}(\s^{n-1}) \,(n-1)^r \, \varepsilon_r'\,(\alpha^*) \,\, .
\end{eqnarray}
Comparing \eqref{variation-good-tris} and \eqref{variation-good-tris+1} we conclude immediately that
\begin{equation}\label{tau-ben-messo-esplicito}
    \tau_r(\varphi_{\alpha^*}) =\,-\, \left [\frac{(n-1)^r}{2} \, \right ]\,\,\varepsilon_r'\,(\alpha^*) \,\, \frac{\partial}{\partial \alpha} \quad {\rm for \,\, all }\quad r \geq 2 \,\, .
\end{equation}
\end{remark}
\begin{remark} An idea similar to Proposition \ref{main-theor-parallels} was used by Hsiang and Lawson (see \cite{HL} and \cite{ER} for related matters) in the context of cohomogeneity zero minimal immersions. In their case, a Lie group $G$ acts isometrically on the ambient space with principal orbits of codimension one, and the search of $G$-invariant minimal immersions reduces to finding critical points of the Volume function which, as in our case, just depends on one real variable.
\end{remark}

\subsection{The proof of Theorem \ref{Corollary-parallel-spheres-Clifford}} This follows steps similar to Theorem \ref{Corollary-parallel-spheres} and so we just point out the relevant modifications. In this case we have to study maps of the following type:
\begin{equation}\label{def-Clifford-tori}\begin{array}{lcll}
                                           \varphi_{\alpha^*} \,:\, &\s^p(R_1) \times \s^q(R_2) &\to &  \s^{p+q+1} \subset \R^{p+1}\times \R^{q+1}\\
                                           &&& \\
                                           &(R_1\,w,R_2\,z) &\mapsto & ( \sin \alpha^*\,w,  \cos \alpha^*\,z) \,\, ,
                                         \end{array}
\end{equation}
where $w,z$ denote the generic coordinates of a point of $\s^p$ and $\s^q$ respectively and $\alpha^* $ is a chosen value in the interval $(0,(\pi \slash 2))$. The first step is to determine when a map of the type \eqref{def-Clifford-tori} is $r$-harmonic. The version of Proposition \ref{main-theor-parallels} in this context is the following:
\begin{proposition}\label{main-theor-Clifford}Let $\varphi_{\alpha^*}\,: \,\,\s^p(R_1) \times \s^q(R_2) \to \,\s^{p+q+1}$ be as in \eqref{def-Clifford-tori}. Then $\varphi_{\alpha^*}$ is a proper $r$-harmonic map ($r \geq 2$) if and only if
\begin{equation}\label{not-harmonic-Clifford}
   \left [\frac{p}{R_1^2} \,-\,\frac{q}{R_2^2} \right ]\, \neq \, 0
\end{equation}
and $\alpha^*$ is a critical point of the function $\varepsilon^C_r \, : (0,(\pi \slash 2)) \to \R $ defined by
\begin{equation}\label{condizione-r-harmonicity-Clifford}
 \varepsilon^C_r (\alpha)= \sin^2 \alpha \, \cos^2 \alpha \,\,\left [\frac{p}{R_1^2}\,\cos^2 \alpha\,+\,\frac{q}{R_2^2}\,\sin^2 \alpha\right ]^{(r-2)} \,\, .
\end{equation}
\end{proposition}
\begin{remark} If \eqref{not-harmonic-Clifford} holds and $r \geq3$, then the explicit form of the condition $(\varepsilon^C_r)' (\alpha^*)=0$ is equivalent to:
\begin{equation}\label{condizione-r-harmonicity-clifford-esplicita}
     \frac{p}{R_1^2}+ \left [ (r-1)\,\left ( \frac{q}{R_2^2}-\frac{p}{R_1^2}\right )-2\,\frac{p}{R_1^2} \right ] \sin^2 \alpha^* + r\,\left [ \frac{p}{R_1^2}\,-\,\frac{q}{R_2^2}\right ]\sin^4 \alpha^* =0 \,\, .
\end{equation}
\end{remark}
\begin{remark}\label{remark-dopo-main-theorem-clifford} Up to the constant 
$$
c=\frac{1}{2}\, {\rm Vol}\left (\s^p(R_1) \times \s^q(R_2)\right )\, \left [\frac{p}{R_1^2} \,-\,\frac{q}{R_2^2}\right ]^2  
$$
$\varepsilon^C_r (\alpha^*)$ coincides with the $r$-energy of the map $\varphi_{\alpha^*}$ in \eqref{def-Clifford-tori} ($r \geq 2$). In particular, if the left member of \eqref{not-harmonic-Clifford} vanishes, then $\varphi_{\alpha^*}$ is harmonic.
\end{remark}
In the case of maps as in \eqref{def-Clifford-tori} the induced pull-back metric identifies the domain with $\s^p(\sin \alpha^* )\times \s^q(\cos \alpha^*)$. Therefore, in order to ensure that an $r$-harmonic map of type \eqref{def-Clifford-tori} is an \textit{isometric} immersion, it is enough to determine the solutions of \eqref{condizione-r-harmonicity-clifford-esplicita} with $R_1^2=\sin^2 \alpha^* $ and $R_2^2=\cos^2 \alpha^*$. More precisely, by setting $R_1^2=\sin^2 \alpha^*=t$, \eqref{condizione-r-harmonicity-clifford-esplicita} becomes equivalent to the fact that $t$ is a root of the polynomial $P(t)$ in \eqref{r-harmonicity-Clifford} and a straightforward inspection leads us to the end of the proof of Theorem \ref{Corollary-parallel-spheres-Clifford}.


\begin{thebibliography}{99}
%\bibitem{BMBib} {\it The Harmonic Maps Bibliography}.\\
%{\tt http://people.bath.ac.uk/masfeb/harmonic.html}

%\bibitem{BMO08} A. Balmu\c s, S. Montaldo, C. Oniciuc.
%Classification results for biharmonic submanifolds in spheres. {\em Israel J. Math.} 168 (2008), 201--220.

\bibitem{CMO02} R.~Caddeo, S.~Montaldo, C.~Oniciuc.
Biharmonic submanifolds in spheres. {\em Israel J. Math.} 130 (2002), 109--123.

\bibitem{Chen} B.-Y.~Chen, {\em Total mean curvature and submanifolds of finite type}. Second edition. Series in Pure Mathematics, 27. World Scientific Publishing Co. Pte. Ltd., Hackensack, NJ, (2015).


\bibitem{EL1} J.~Eells, L.~Lemaire. { Another report on harmonic maps}.
 {\it Bull. London Math. Soc.},  20 (1988), 385--524.

\bibitem{EL83} J.~Eells, L.~Lemaire. {\it  Selected topics in harmonic maps.} CBMS Regional Conference Series in Mathematics, 50. American Mathematical Society, Providence, RI, 1983.

\bibitem{ER} J.~Eells, A.~Ratto. {\it Harmonic Maps and Minimal Immersions with Symmetries.
Methods of Ordinary Differential Equations Applied to Elliptic
Variational Problems}. Annals of Mathematics Studies (133), Princeton University Press, (1993).

\bibitem{ES} J.~Eells, J.H.~Sampson. {Variational theory in fibre bundles}. {\em Proc. U.S.-Japan Seminar in Differential Geometry}, Kyoto (1965), 22--33.

\bibitem{HL} W.-Y. Hsiang, B.~Lawson. Minimal submanifolds of low cohomogeneity. {\em J. Diff. Geom.} 5 (1971), 1--38.

\bibitem{Jiang} G.Y.~Jiang. {2-harmonic maps and their first and second variation formulas}.
{\it Chinese Ann. Math. Ser. A 7},  7 (1986), 130--144.

\bibitem{Maeta1} S.~Maeta. {k-harmonic maps into a Riemannian manifold with constant sectional curvature}.
{\it Proc. A.M.S.},  140 (2012), 1835--1847.

\bibitem{Maeta2} S.~Maeta. {Construction of triharmonic maps}.
{\it Houston J. Math.} 41 (2015), 433--444. 

\bibitem{Maeta3} S.~Maeta. {The second variational formula of the $k$-energy and $k$-harmonic curves}. {\it Osaka J. Math.} 49 (2012), 1035--1063.

\bibitem{Maeta4} S.~Maeta, N.~Nakauchi, H.~Urakawa. {Triharmonic isometric immersions into a manifold of non-positively constant curvature}. {\it Monatsh. Math.} 177 (2015), 551--€"567. 

\bibitem{SMCO} S.~Montaldo, C.~Oniciuc. {A short survey on biharmonic maps between riemannian manifolds}. {\it Rev. Un. Mat. Argentina}, 47 (2006), 1--22.

\bibitem{MOR1} S.~Montaldo, C.~Oniciuc, A.~Ratto. {Proper biconservative immersions into the Euclidean space} {\em Ann. Mat. Pura e Appl.} 195 (2016), 403--422.

\bibitem{Mont-Ratto2} S.~Montaldo, A.~Ratto. Biharmonic submanifolds into ellipsoids. {\em Monatshefte f\"{u}r Mathematik} 176 (2015), 589--601.

\bibitem{Mont-Ratto3} S.~Montaldo, A.~Ratto. A general approach to equivariant biharmonic maps. {\em Med. J. Math.} 10 (2013), 1127--1139.

\bibitem{Palais} R.S.~Palais {The principle of symmetric criticality} {\it Comm. Math. Phys.} 69 (1979), 19--30.

\bibitem{Wang} S.B.~Wang. {The first variation formula for k-harmonic mappings}.
{\it Journal of Nanchang University} 13, N.1 (1989). 

\end{thebibliography}
\end{document}